\documentclass{amsart}

\usepackage{amssymb}
\usepackage{amsmath}
\usepackage{amsthm}
\usepackage{amsfonts}

\usepackage{latexsym}

\usepackage[dvips]{epsfig}
\usepackage{amsbsy}
\usepackage{amsgen}
\usepackage{amscd}
\usepackage{amsopn}
\usepackage{amstext}
\usepackage{amsxtra}
\usepackage{xypic}

\newtheorem{theorem}{Theorem}[section]
\newtheorem{definition}[theorem]{Definition}

\newtheorem{remark}[theorem]{Remark}
\newtheorem{lem}[theorem]{Lemma}
\newtheorem{proposition}[theorem]{Proposition}

\begin{document}
\title[Constructing model categories with prescribed fibrant objects]
{Constructing model categories with prescribed fibrant objects}
\author[{A. E.} {Stanculescu}]{{Alexandru E.} {Stanculescu}}
\address{\newline Department of Mathematics and Statistics,
\newline Masaryk University,  Kotl\'{a}{\v{r}}sk{\'{a}} 2,\newline
611 37 Brno, Czech Republic}
\email{stanculescu@math.muni.cz}
\thanks{Supported by the project CZ.1.07/2.3.00/20.0003
of the Operational Programme Education for Competitiveness of 
the Ministry of Education, Youth and Sports of the Czech Republic
\newline
\indent}

\begin{abstract}
We present a weak form of a recognition 
principle for Quillen model categories due to 
J.H. Smith. We use it to put a model category 
structure on the category of small categories 
enriched over a suitable monoidal simplicial model 
category. The proof uses a part of the model structure 
on small simplicial categories due to J. Bergner.
We give an application of the weak form of Smith's 
result to left Bousfield localizations of categories 
of monoids in a suitable monoidal model category.
\end{abstract}

\maketitle

There are nowadays several recognition principles 
that allow one to put a Quillen model category 
structure on a given category. For the purposes
of this work we divide them into those that make 
use of the small object agument and those that 
don't. A recognition principle that makes use 
of the small object argument is the following 
theorem of J.H. Smith.

\begin{theorem}\cite[Theorem 1.7]{Bek} Let $\mathcal{E}$ 
be a locally presentable category, {\rm W} a full accessible 
subcategory of $\mathrm{Mor}(\mathcal{E})$, and $I$ a 
set of morphisms of $\mathcal{E}$. Suppose they satisfy:

$c0:$ $\mathrm{W}$ has the two out of three property.

$c1:$ $\mathrm{inj}(I)\subset{\rm W}$.

$c2:$ The class ${\rm cof}(I)\cap {\rm W}$ is closed under
transfinite composition and under pushout.

Then setting weak equivalences:={\rm W},
cofibrations:=$\mathrm{cof}(I)$ and
fibrations:=$\mathrm{inj}(\mathrm{cof}(I)\cap {\rm W})$, 
one obtains a cofibrantly generated model structure 
on $\mathcal{E}$.
\end{theorem}
We can say that $(a)$ in practice, it is condition $c2$ above 
that is often the most difficult to check and $(b)$ the result 
gives \emph{no} description of the fibrations of the 
resulting model structure. Another recognition principle 
that makes use of the small object argument is a
result of D.M. Kan \cite[Theorem 11.3.1]{Hi}, 
\cite[Theorem 2.1.19]{Ho}. We can say that Kan's 
result gives a \emph{full} description of the fibrations 
of the resulting model structure. In this paper we

(1) advertise (see Proposition 1.3) an abstraction 
of a technique due to D.-C. Cisinski 
\cite[Proof of Th\'{e}or\`{e}me 1.3.22]{Ci} and 
A. Joyal (unpublished, but present in his proof, circa 1996, 
of the model structure for quasi-categories) that addresses 
both $(a)$ and $(b)$ above, in the sense that it makes $c2$ 
easier to check and it gives a \emph{partial} description 
of the fibrations of the resulting model structure---namely
the fibrant objects and the fibrations between them are 
described---provided that other assumptions hold, and 

(2) give an application of this technique to the homotopy 
theory of categories enriched over a suitable monoidal 
simplicial model category (see Theorem 2.3) and to left 
Bousfield localizations of categories of monoids in a 
suitable monoidal model category (see Theorem 4.5).

The paper is organized as follows. In Section 1 we detail 
the above mentioned technique. The two out of six 
property of a class of maps of Dwyer et al. \cite{DHKS} 
plays an important role. In Section 2 we prove that 
the category of small categories enriched over a 
monoidal simplicial model category that satisfies
some assumptions, admits a certain model category 
structure. Our proof uses one result of the non-formal 
part of the proof of the analogous model structure for 
categories enriched over the category of simplicial sets, 
due to J. Bergner \cite{Be}. We modify one of the steps 
in Bergner's proof; this modification is a key point in our 
approach and it enables us to apply the technique from 
Section 1. We also fix (see Remark 2.8), in an appropiate 
way, a mistake in \cite{St}. The idea to use the model 
structure for categories enriched over the category of 
simplicial sets is due to G. Tabuada \cite{Ta}. 
In Section 3 we extend a result of R. Fritsch 
and D.M. Latch \cite[Proposition 5.2]{FL} 
to enriched categories; this is needed in the proof of the 
main result of Section 2. The section is self contained. 
Motivated by considerations from \cite{Ho2}, we apply 
in Section 4 the technique from Section 1 to the study 
of left Bousfield localizations of categories of monoids. 
Precisely, let $L${\bf M} be a left Bousfield localization 
of a monoidal model category {\bf M}. We consider 
the problem of putting a model 
category structure on the category of monoids 
in {\bf M}, somehow related to $L${\bf M}.

\section{Constructing model categories with 
prescribed fibrant objects}

We recall from \cite{DHKS} the following definitions. Let $\mathcal{E}$ 
be an arbitrary category and {\rm W} a class of maps of $\mathcal{E}$. 
{\rm W} is said to satisfy the \emph{two out of six property} if for every 
three maps $r,s,t$ of $\mathcal{E}$ for which the two compositions 
$sr$ and $ts$ are defined and are in {\rm W}, the four maps 
$r,s,t$ and $tsr$ are in {\rm W}. {\rm W} is said to satisfy the 
\emph{weak invertibility property} if every map $s$ of $\mathcal{E}$ for 
which there exist maps $r$ and $t$ such that the compositions $sr$ and 
$ts$ exist and are in {\rm W}, is itself in {\rm W}. The two out of six 
property implies the two out of three property. The converse holds 
in the presence of the weak invertibility property.

The terminal object of a category, when it exists, 
is denoted by $1$.

Let $\mathcal{E}$ be a locally presentable 
category and $J$ a set of maps of $\mathcal{E}$. 
Then the pair $({\rm cof}(J),{\rm inj}(J))$
is a weak factorization system on $\mathcal{E}$
\cite[Proposition 1.3]{Bek}.
We call a map of $\mathcal{E}$ that belongs 
to ${\rm inj}(J)$ a \emph{naive fibration}, 
and say that an object $X$ of $\mathcal{E}$ 
is \emph{naively fibrant} if $X\rightarrow 1$ 
is a naive fibration. We denote the class of
naive fibrations between naively fibrant objects
by ${\rm inj}_{0}(J)$.
\begin{lem} {\rm (D.-C. Cisinski, A. Joyal)}
Let $\mathcal{E}$ be a locally presentable category, 
$(\mathcal{A},\mathcal{B})$ a weak factorisation 
system on $\mathcal{E}$, {\rm W} a class of maps 
of $\mathcal{E}$ satisfying the two out of six property 
and $J$ a set of maps of $\mathcal{E}$. 

$(1)$ Suppose that $cell(J)\subset {\rm W}$. Then a 
map that has the left lifting property with respect to 
maps in ${\rm inj}_{0}(J)$ belongs to {\rm W}.

$(2)$ Suppose that $cell(J)\subset {\rm W}$ and that
${\rm inj}_{0}(J)\cap{\rm W}\subset\mathcal{B}$. 
Then a map in $\mathcal{A}$ belongs to {\rm W} 
if and only if it has the left lifting property with 
respect to the maps in ${\rm inj}_{0}(J)$.
In particular, $\mathcal{A}\cap {\rm W}$ is 
closed under pushouts and transfinite compositions.

$(3)$ Suppose that $cell(J)\subset \mathcal{A}\cap{\rm W}$ 
and that ${\rm inj}_{0}(J)\cap{\rm W}\subset\mathcal{B}$.
Then an object $X$ of $\mathcal{E}$ is naively fibrant 
if and only if the map $X\rightarrow 1$ is in 
$\mathrm{inj}(\mathcal{A}\cap {\rm W})$.
Also, a map between naively fibrant objects is in 
$\mathrm{inj}(\mathcal{A}\cap {\rm W})$ if and 
only if it is a naive fibration.
\end{lem}
\begin{proof}
$(1)$ Let $i:A\rightarrow B$ be a map which has the 
left lifting property with respect to the naive fibrations 
between naively fibrant objects. Factorize (see, 
for example, \cite[Proposition 1.3]{Bek}) the map 
$B\rightarrow 1$ as $B\rightarrow \bar{B}\rightarrow 1$, 
where $B\rightarrow \bar{B}$ is in $cell(J)$ 
and $\bar{B}$ is naively fibrant. Next, factorize the 
composite map $A\rightarrow \bar{B}$ as a map
$A\rightarrow \bar{A}$ in $cell(J)$ 
followed by a naive fibration $\bar{A} \rightarrow 
\bar{B}$. The resulting commutative diagram
\[
\xymatrix{
A \ar[r] \ar[d]_{i} & \bar{A} \ar[d]\\
B \ar[r] & \bar{B}\\
}
   \]
has then a diagonal filler, and so the hypothesis 
and the two out of six property of {\rm W} imply 
that $i$ is in {\rm W}. 

$(2)$ Let 
\[
   \xymatrix{
A \ar[r]^{u} \ar[d]_{i} & X \ar[d]^{p}\\
B \ar[r]^{v} & Y\\
  }
  \]
be a commutative diagram with $i$ in 
$\mathcal{A}\cap {\rm W}$ and $p$ in 
${\rm inj}_{0}(J)$. Factorize 
$v$ as a map $B\rightarrow \bar{B}$ in $cell(J)$ 
followed by a naive fibration $\bar{B}\rightarrow Y$. 
Next, factorize the canonical map 
$A\rightarrow \bar{B}\underset{Y}\times X$ 
as a map $A\rightarrow \bar{A}$ in $cell(J)$ 
followed by a naive fibration 
$\bar{A}\rightarrow \bar{B}\underset{Y}\times X$. 
It suffices to show that the square
\[
   \xymatrix{
A \ar[r] \ar[d]_{i} & \bar{A} \ar[d]\\
B \ar[r] & \bar{B}\\
}
   \]
has a diagonal filler. The map $\bar{A}\rightarrow \bar{B}$ 
is a naive fibration between naively fibrant objects.
It also belongs to {\rm W} by the two out of three property, 
and so by hypothesis it is in $\mathcal{B}$. 
Therefore the diagonal filler exists. 
The converse follows from $(1)$. Thus, in order to 
detect if an element of $\mathcal{A}$ is in
{\rm W} one can use the left lifting property with 
respect to a class of maps, namely ${\rm inj}_{0}(J)$. 
In particular, $\mathcal{A}\cap {\rm W}$
is closed under pushouts and transfinite compositions.

$(3)$ This is straightforward from $(2)$.
\end{proof}
\begin{remark}
{\rm One can make variations in Lemma 1.1. For example, 
the path object argument devised by Quillen shows that the 
conclusion of $(1)$ remains valid if instead of
$cell(J)\subset {\rm W}$ one requires that}
$\mathcal{E}$ has a functorial naively fibrant replacement 
functor and every naively fibrant object has a naive path object. 
{\rm This new requirement implies that
$cell(J)\subset {\rm W}$.}
\end{remark}
The following result makes the connection 
between Smith's Theorem and Lemma 1.1.
\begin{proposition}
Let $\mathcal{E}$ be a locally presentable category, 
{\rm W} a full accessible subcategory of
$\mathrm{Mor}(\mathcal{E})$ and $I$ and 
$J$ be two sets of morphisms of $\mathcal{E}$. 
Let us call a map of $\mathcal{E}$ that belongs to 
${\rm inj}(J)$ a \emph{naive fibration}, 
and an object $X$ of $\mathcal{E}$ 
\emph{naively fibrant} if $X\rightarrow 1$ is 
a naive fibration. Suppose the following
conditions are satisfied:

$c0:$ $\mathrm{W}$ has the two out of three property.

$c1:$ $\mathrm{inj}(I)\subset{\rm W}$.

$nc0:$ {\rm W} has the weak invertibility property.

$nc1:$ $cell(J)\subset {\rm cof}(I)\cap {\rm W}$.

$nc2:$ A map between naively fibrant objects that is both 
a naive fibration and in {\rm W} is in {\rm inj}$(I)$.

Then the triple $({\rm W}, {\rm cof}(I), 
{\rm inj}({\rm cof}(I)\cap {\rm W}))$
is a model structure on $\mathcal{E}$. 
Moreover, an object of $\mathcal{E}$ is 
fibrant if and only if it is naively fibrant, 
and the fibrations between fibrant 
objects are the naive fibrations.
\end{proposition}
\begin{proof}
We shall use Theorem 0.1. All
the assumptions of this result hold,
except possibly condition $c2$.
To check that $c2$ holds we apply the last part of
Lemma 1.1(2) to the weak factorization system 
$(\mathcal{A},\mathcal{B})=({\rm cof}(I),{\rm inj}(I))$. 
It follows that the triple $({\rm W}, {\rm cof}(I), 
{\rm inj}({\rm cof}(I)\cap {\rm W}))$
is a model structure. The characterization of 
fibrant objects and of the fibrations between 
fibrant objects is then a consequence 
of Lemma 1.1(3) applied to 
the weak factorization system 
$({\rm cof}(I),{\rm inj}(I))$.
\end{proof}
The following result is a variation of
Proposition 1.3, essentially due to 
A.K. Bousfield \cite[Proof of Theorem 9.3]{Bo}.
We leave the proof to the interested reader.
\begin{proposition}
Let $\mathcal{E}$ be a category that is
closed under limits and colimits and let
{\rm W} be a class of maps of $\mathcal{E}$
that has the two out of three property. If
$I$ and $J$ are two sets of morphisms of 
$\mathcal{E}$ such that

$(1)$ both $I$ and $J$ permit the small object
argument \cite[Definition 10.5.15]{Hi},

$(2)$ ${\rm inj}(I)\subset{\rm W}$,

$(3)$ $cell(J)\subset {\rm cof}(I)\cap {\rm W}$,

$(4)$ ${\rm inj}_{0}(J)\cap{\rm W}\subset {\rm inj}(I)$, 
and

$(5)$ the class {\rm W} is stable under pullback
along maps in ${\rm inj}_{0}(J)$,

then the triple $({\rm W}, {\rm cof}(I), 
{\rm inj}({\rm cof}(I)\cap {\rm W}))$
is a right proper model structure on $\mathcal{E}$. 
Moreover, an object of $\mathcal{E}$ is 
fibrant if and only if it is naively fibrant, 
and the fibrations between fibrant 
objects are the naive fibrations.
\end{proposition}
Here is an application of Lemma 1.1. Let $\mathcal{E}$ 
be a locally presentable closed category with initial object
$\emptyset$. We denote by $\otimes$ the monoidal product 
of $\mathcal{E}$ and for two objects $X,Y$ of $\mathcal{E}$ 
we write $Y^{X}$ for their internal hom. In the language of 
Lemma 1.1 we have
\begin{proposition}
Let {\rm W} be a class of maps of $\mathcal{E}$ having 
the two out of six property and let $I$ and $J$
be two sets of maps of $\mathcal{E}$. Suppose that
the domains of the elements of $I$ are in {\rm cof}$(I)$,
that $cell(J)\subset {\rm cof}(I)\cap {\rm W}$ 
and that a map between naively fibrant objects which 
is both a naive fibration and in {\rm W} is in {\rm inj}$(I)$.
Then the following are equivalent:

$(a)$ for any maps $A\rightarrow B$ and $K\rightarrow L$ 
of {\rm cof}$(I)$, the canonical map 
$$A\otimes L\underset{A\otimes K}\cup B\otimes 
K\rightarrow B\otimes L$$ is in {\rm cof}$(I)$, which 
is in {\rm W} if either one of the given maps is in {\rm W};

$(b)$ for any maps $A\rightarrow B$ and $K\rightarrow L$ 
of {\rm cof}$(I)$, the canonical map 
$$A\otimes L\underset{A\otimes K}\cup B\otimes 
K\rightarrow B\otimes L$$ is in {\rm cof}$(I)$ \emph{and}
for every element $A\rightarrow B$ of $I$ and every naive 
fibration $X\rightarrow Y$ between naively fibrant objects, 
the canonical map $$X^{B}\rightarrow 
Y^{B}\times_{Y^{A}}X^{A}$$
is a naive fibration between naively fibrant objects.
\end{proposition}
\begin{proof}
The fact that the domains of the elements of $I$ 
are in {\rm cof}$(I)$ means that for every element 
$A\rightarrow B$ of $I$, the map $\emptyset \to A$
is in {\rm cof}$(I)$ (and therefore so is $\emptyset \to B$).

We prove $(a)\Rightarrow (b)$. Let $A\rightarrow B$ 
be an element of $I$, $X\rightarrow Y$ a naive fibration
between naively fibrant objects and $C\rightarrow D$ 
an element of $J$. A commutative diagram
\[
\xymatrix{
C \ar[rr] \ar[d] & & X^{B} \ar[d]\\
D \ar[rr] & & Y^{B}\times_{Y^{A}}X^{A}\\
  }
  \]
has a diagonal filler if and only if its
adjoint transpose 
\[
   \xymatrix{
C\otimes B\underset{C\otimes A}\cup 
D\otimes A \ar[rr] \ar[d] & & X \ar[d]\\
D\otimes B \ar[rr] & & Y\\
  }
  \]
has one. The latter is true by Lemma 1.1(2)
applied to the weak factorization system 
$({\rm cof}(I),{\rm inj}(I))$. It follows that
$X^{B}\rightarrow Y^{B}\times_{Y^{A}}X^{A}$
is a naive fibration. A similar adjunction argument 
shows that $X^{A}\rightarrow Y^{A}$ and 
$X^{B}\rightarrow Y^{B}$ are naive fibrations
between naively fibrant objects, therefore
$Y^{B}\times_{Y^{A}}X^{A}$ is naively
fibrant.

We prove $(b)\Rightarrow (a)$. Suppose first 
that $A\rightarrow B$ is an element of $I$
and let $K\rightarrow L$ be a fixed map in
${\rm cof}(I)\cap {\rm W}$. Then the 
canonical map 
$$A\otimes L\underset{A\otimes K}\cup B\otimes 
K\rightarrow B\otimes L$$ is in {\rm W}
by Lemma 1.1(2) applied to the weak 
factorization system $({\rm cof}(I),{\rm inj}(I))$
and an adjunction argument. Thus,
it suffices to show that the class of maps
$A'\rightarrow B'$ of {\rm cof}$(I)$ such that 
$$A'\otimes L\underset{A'\otimes K}\cup 
B'\otimes K\rightarrow B'\otimes L$$ 
is in {\rm cof}$(I)\cap {\rm W}$ is closed 
under pushout, transfinite composition and
retracts. This is the case since by Lemma 1.1(2)
applied to the weak factorization system 
$({\rm cof}(I),{\rm inj}(I))$ the elements of
${\rm cof}(I)$ which are in {\rm W}
can be detected by the left lifting property 
with respect to a class of maps.
\end{proof}

\section{Application: categories enriched over 
monoidal simplicial model categories}

We denote by {\bf S} the category of simplicial 
sets, regarded as having the standard 
model structure (due to Quillen). We let 
{\bf Cat} be the category of small categories. 
We say that an arrow 
$f:C\rightarrow D$ of {\bf Cat} is an 
\emph{isofibration} if for any $x\in Ob(C)$ and any 
isomorphism $v:y'\rightarrow f(x)$ in $D$, there 
exists an isomorphism $u:x'\rightarrow x$ in $C$ 
such that $f(u)=v$. The class of isofibrations is 
invariant under isomorphisms in the sense that 
given a commutative diagram in {\bf Cat}
\[
\xymatrix{
{A}\ar[r] \ar[d]_{f}
&{B} \ar[d]^{g}\\
{C} \ar[r]
&{D}
}
\]
in which the horizontal arrows are isomorphisms,
the map $f$ is an isofibration if and only if $g$ is so.

\subsection{Monoidal simplicial model categories}

Let {\bf M} be a monoidal model category 
with cofibrant unit. We recall \cite[Definition 4.2.20]{Ho}
that {\bf M} is said to be a \emph{monoidal} 
{\bf S}-\emph{model category} if it is given
a Quillen pair $F:{\bf S}\rightleftarrows {\bf M}:G$ 
such that $F$ is strong monoidal. Since $F$
is strong monoidal, $G$ becomes a monoidal 
functor.

\subsection{Classes of {\bf M}-functors and the main result}

Let {\bf M} be a monoidal model category with cofibrant unit $e$.
We denote by {\bf M}-{\bf Cat} the category of small {\bf M}-categories.
If $S$ is a set, we denote by {\bf M}-{\bf Cat}$(S)$ 
(resp. {\bf M}-{\bf Graph}$(S)$) the category of small 
{\bf M}-categories (resp. {\bf M}-graphs) with fixed set of 
objects $S$. When $S$ is a one element set $\{\ast\}$,
{\bf M}-{\bf Cat}$(\{\ast\})$ is the category 
$Mon({\bf M})$ of monoids in {\bf M}.
There is a free-forgetful adjunction 
$$F_{S}:{\bf M}\text{-}{\bf Graph}(S)\rightleftarrows 
{\bf M}\text{-}{\bf Cat}(S):U_{S}$$
We denote by $\varepsilon^{S}$ the counit of this adjunction. 
Every function $f:S\rightarrow T$ induces an adjoint pair
$$f_{!}:{\bf M}\text{-}{\bf Cat}(S)\rightleftarrows 
{\bf M}\text{-}{\bf Cat}(T):f^{\ast}$$
If $\mathcal{K}$ is a class of maps of {\bf M}, 
an {\bf M}-functor $f:\mathcal{A}\rightarrow \mathcal{B}$ 
is said to be \emph{locally in} $\mathcal{K}$ if for each pair 
$x,y\in \mathcal{A}$ of objects, the map 
$f_{x,y}:\mathcal{A}(x,y)\rightarrow \mathcal{B}(f(x),f(y))$ 
is in $\mathcal{K}$.

We have a functor $[\_]_{{\bf M}}:{\bf M}\text{-}{\bf Cat} 
\rightarrow {\bf Cat}$ obtained by change of base along 
the symmetric monoidal composite functor
\[
   \xymatrix{
   {\bf M} \ar[r] & Ho({\bf M})
   \ar[rr]^{Ho({\bf M})(e,-)} & & Set\\
}
\]
\begin{definition} Let $f:\mathcal{A}\rightarrow \mathcal{B}$ 
be a morphism in {\bf M}\text{-}{\bf Cat}.

1. The morphism $f$ is a $\mathrm{DK}$-\emph{equivalence} if
$f$ is locally a weak equivalence of {\bf M} and
$[f]_{{\bf M}}:[\mathcal{A}]_{{\bf M}}\rightarrow [\mathcal{B}]_{{\bf M}}$
is essentially surjective.

2. The morphism $f$ is a $\mathrm{DK}$-\emph{fibration} 
if $f$ is locally a fibration of {\bf M} and $[f]_{{\bf M}}$ 
is an isofibration.

3. The morphism $f$ is called a \emph{trivial fibration} 
if it is both a DK-equivalence and a DK-fibration.

4. The morphism $f$ is called a \emph{cofibration} 
if it has the left lifting property with respect to the trivial 
fibrations.
\end{definition}
It follows from Definition 2.1 that $(a)$ an {\bf M}-functor
$f$ is a DK-equivalence if and only if $Ho(f)$ is  an equivalence 
of $Ho({\bf M})$-categories, and $(b)$ an {\bf M}-functor is a
trivial fibration if and only if it is surjective on objects and 
locally a trivial fibration of {\bf M}. In particular, the
class of DK-equivalences has the two out of three 
and weak invertibility properties.

We denote by $\mathcal{I}$ the {\bf M}-category with a single object $\ast$
and $\mathcal{I}(\ast,\ast)=e$. For an object $X$ of {\bf M} we denote by 
$2_{X}$ the {\bf M}-category with two objects 0 and 1 and with 
$2_{X}(0,0)=2_{X}(1,1)=e$, $2_{X}(0,1)=X$ and $2_{X}(1,0)=\emptyset$.
When {\bf M} is cofibrantly generated, an {\bf M}-functor is a trivial 
fibration if and only if it has the right lifting property with respect to the
saturated class generated by $\{\emptyset\rightarrow \mathcal{I}\}\cup 
\{2_{X}\overset{2_{i}} \rightarrow 2_{Y}$, $i$ generating cofibration  
of {\bf M}\}, where $\emptyset$ denotes the initial object of 
{\bf M}\text{-}{\bf Cat}. We have the following fundamental 
result of J. Bergner.
\begin{theorem} \cite{Be} The category
{\bf S}\text{-}{\bf Cat} of simplicial categories admits a cofibrantly
generated model structure in which the weak equivalences are the
DK-equivalences and the fibrations are the DK-fibrations. A
generating set of trivial cofibrations consists of

(B1) $\{2_{X}\overset{2_{j}} \longrightarrow 2_{Y}\}$, 
where $j$ is a horn inclusion, and

(B2) inclusions $\mathcal{I} \overset{\delta_{y}} \rightarrow
\mathcal{H}$, where $\{\mathcal{H}\}$ is a set of representatives
for the isomorphism classes of simplicial categories on two objects
which have  countably many simplices in each function complex.
Furthermore, each such $\mathcal{H}$ is required to be cofibrant 
and weakly contractible in ${\bf S}\text{-}{\bf Cat}(\{x,y\})$. 
Here $\{x,y\}$ is the set with elements $x$ and $y$ and 
$\delta_{y}$ omits $y$.
\end{theorem}
Recall from \cite[Definition 3.3]{SS1} the monoid 
axiom. The main result of this section is
\begin{theorem}
Let {\bf M} be a cofibrantly generated monoidal 
{\bf S}-model category having cofibrant unit 
and which satisfies the monoid axiom. Suppose 
furthermore that {\bf M} is locally presentable
and that a transfinite composition of weak 
equivalences of {\bf M} is a weak equivalence.

Then {\bf M}\text{-}{\bf Cat} admits a 
cofibrantly generated model category
structure in which the weak equivalences 
are the DK-equivalences, the cofibrations 
are the elements of 
{\rm cof}$(\{\emptyset\rightarrow \mathcal{I}\}\cup 
\{2_{X}\overset{2_{i}} \rightarrow 2_{Y}$, $i$ 
generating cofibration  of {\bf M}\}{\rm )},
the fibrant objects are the locally fibrant 
{\bf M}-categories and the fibrations
between fibrant objects are the DK-fibrations. 

If the model structure on {\bf M} is right proper, 
then so is the one on {\bf M}\text{-}{\bf Cat}.
\end{theorem}
\begin{proof}
We shall apply Theorem 0.1 via Proposition 1.3. 
We take $\mathcal{E}$ to be {\bf M}\text{-}{\bf Cat}
and {\rm W} to be the class of DK-equivalences.
The fact that {\bf M}\text{-}{\bf Cat} is locally 
presentable can be seen in a few ways, one is 
presented in \cite{KL}. The fact that the class of 
DK-equivalences is accessible follows essentially 
from the fact that the classes of weak equivalences 
of {\bf M} and of essentially surjective functors 
are accessible. We take $I$ to be the set 
$\{\emptyset\rightarrow \mathcal{I}\}\cup \{2_{X}
\overset{2_{i}} \rightarrow 2_{Y}$, 
$i$ generating cofibration of {\bf M}\}.
Let $$F:{\bf S}\rightleftarrows {\bf M}:G$$ be the 
Quillen pair guaranteed by the definition. $(F,G)$ 
induces adjoint pairs $$F':{\bf S}\text{-}{\bf Cat} 
\rightleftarrows {\bf M}\text{-}{\bf Cat}:G'$$ 
and $$F':{\bf S}\text{-}{\bf Cat}(S) 
\rightleftarrows {\bf M}\text{-}{\bf Cat}(S):G'$$ 
for every set $S$. The first $G'$ functor preserves 
trivial fibrations and the {\bf M}-functors
which are locally a fibration. The latter adjoint 
pair is a Quillen pair. Finally, we take $J$ to be the set 
$F'(B2)\cup \{2_{X}\overset{2_{i}} \rightarrow 2_{Y}$, 
$i$ generating trivial cofibration  of {\bf M}\},
where $B2$ is as in Theorem 2.2.

\emph{Step 1}. Conditions $c0, c1$ and $nc0$ 
from Proposition 1.3 were dealt with above.

\emph{Step 2}. Since every map $\delta_{y}$ 
belonging to the set B2 from Theorem 2.2 
has a retraction, one readily checks that an 
{\bf M}-category is naively fibrant if and 
only if it is locally fibrant. We claim
that if an {\bf M}-functor between locally 
fibrant {\bf M}-categories is a naive 
fibration, then it is a DK-fibration.
To see this, let first {\bf M}\text{-}{\bf Cat}$_{f}$ 
be the full subcategory of {\bf M}\text{-}{\bf Cat} 
consisting of the locally fibrant {\bf M}-categories. 
By \cite[Proposition 8.5.16]{Hi} 
we have a natural isomorphism of functors
$$\eta:[\_]_{{\bf S}}G'\cong [\_]_{{\bf M}}:
{\bf M}\text{-}{\bf Cat}_{f}\rightarrow {\bf Cat}$$
such that for all $\mathcal{A} \in 
{\bf M}\text{-}{\bf Cat}_{f}$, $\eta_{\mathcal{A}}$
is the identity on objects: indeed, for each pair 
$x,y\in \mathcal{A}$ of objects we 
have natural isomorphisms 
$$[\mathcal{A}]_{{\bf M}}(x,y)\cong
Ho({\bf M})(e,\mathcal{A}(x,y))\cong
Ho({\bf M})(F1,\mathcal{A}(x,y))\cong
Ho({\bf S})(1,G\mathcal{A}(x,y))\cong
[G'\mathcal{A}]_{{\bf S}}(x,y)$$
Second, we use the following 
relaxed version of \cite[Proposition 2.3]{Be}.
Let $f$ be a simplicial functor between categories 
enriched in Kan complexes such that $f$ is locally 
a Kan fibration. If $f$ has the right lifting property 
with respect to every element of the set B2, 
then $f$ is a DK-fibration.
(This is the only fact from \cite{Be} that we need.)
These facts, together with the observation that 
the class of isofibrations is invariant in {\bf Cat}
under isomorphisms, imply the claim. It is now clear 
that condition $nc2$ from Proposition 1.3 holds.

\emph{Step 3}. We check condition $nc1$ 
from Proposition 1.3. Let $j:X\rightarrow Y$ be a 
trivial cofibration of {\bf M}. We show that
for every {\bf M}-category $\mathcal{A}$, in 
the pushout diagram
\[
   \xymatrix{
2_{X} \ar[r]^{2_{j}} \ar[d] & 2_{Y} \ar[d]\\
\mathcal{A} \ar[r] & \mathcal{B}\\
}
   \]
the map $\mathcal{A}\rightarrow \mathcal{B}$
is a DK-equivalence. Let $S=Ob(\mathcal{A})$. 
This pushout can be calculated as the pushout 
\[
   \xymatrix{
F_{S}U_{S}\mathcal{A} \ar[r] 
\ar[d]^{\varepsilon^{S}_{\mathcal{A}}}
 & F_{S}\mathcal{X} \ar[d]\\
\mathcal{A} \ar[r] & \mathcal{B}\\
}
   \]
where $U_{S}\mathcal{A}\rightarrow \mathcal{X}$ 
is a certain  map of {\bf M}-graphs with fixed set of 
objects $S$. But then the map $\mathcal{A}
\rightarrow \mathcal{B}$ is known to be locally 
a weak equivalence of {\bf M}, see 
\cite[Proof of Proposition 6.3(1)]{SS2}. 

We now claim that if 
$\delta_{y}:\mathcal{I}\rightarrow \mathcal{H}$ 
is a map belonging to the set B2 from Theorem 2.2 
and $\mathcal{A}$ is any {\bf M}-category, 
then in the pushout diagram
\[
\xymatrix{
F'\mathcal{I} \ar[r]^{a} \ar[d]_{F'\delta_{y}} & \mathcal{A} \ar[d]\\
F'\mathcal{H} \ar[r] & \mathcal{B}\\
}
  \]
the map $\mathcal{A}\rightarrow \mathcal{B}$ is a 
DK-equivalence. We factorize the map $\delta_{y}$ as $\mathcal{I}
\overset{\delta_{y}'} \longrightarrow \mathcal{H}'
\rightarrow \mathcal{H}$, where the simplicial category 
$\mathcal{H}'$ has $\{x\}$ as set of objects and 
$\mathcal{H}'(x,x)=\mathcal{H}(x,x)$,
and then we take consecutive pushouts:
\[
\xymatrix{
F'\mathcal{I} \ar[r]^{a} \ar[d]_{F'\delta_{y}'} & \mathcal{A} \ar[d]^{j}\\
F'\mathcal{H}' \ar[d] \ar[r] & \mathcal{A'} \ar[d]\\
F'\mathcal{H} \ar[r] & \mathcal{B}\\
}
  \]
The map $j$ can be obtained
from the pushout diagram in 
${\bf M}\text{-}{\bf Cat}(Ob(\mathcal{A}))$
\[
   \xymatrix{
a_{!}F'\mathcal{I} \ar[r] \ar[d]_{a_{!}F'\delta_{y}'} 
& \mathcal{A} \ar[d]^{j}\\
a_{!}F'\mathcal{H}' \ar[r] & \mathcal{A}'\\
}
   \]
where $a_{!}:Mon({\bf M})\rightarrow 
{\bf M}\text{-}{\bf Cat}(Ob(\mathcal{A}))$.
By Lemma 2.4 the map $\delta_{y}'$ is a trivial 
cofibration in the category of simplicial monoids,
therefore $F'\delta_{y}'$ is a trivial cofibration
in the category of monoids in {\bf M}. Since
$a_{!}$ is a left Quillen functor, $j$ is a
trivial cofibration in ${\bf M}\text{-}{\bf Cat}(Ob(\mathcal{A}))$. 

The map $F'\mathcal{H}'\rightarrow F'\mathcal{H}$
is a full and faithful inclusion, so by Proposition 3.1 
the map $\mathcal{A}'\rightarrow \mathcal{B}$ 
is a full and faithful inclusion. Therefore 
the map $\mathcal{A}\rightarrow \mathcal{B}$ 
is locally a weak equivalence of {\bf M}. Applying 
the functor $[\_]_{{\bf M}}$ to the diagram 
\[
\xymatrix{
F'\mathcal{I} \ar[r]^{a} \ar[d]_{F'\delta_{y}} & \mathcal{A} \ar[d]\\
F'\mathcal{H} \ar[r] & \mathcal{B}\\
}
  \]
and taking into account that $F'$ preserves DK-equivalences
and that $Ob(\mathcal{B})=Ob(\mathcal{A})\cup \{\ast\}$, 
it follows that $\mathcal{A}\rightarrow \mathcal{B}$ 
is a DK-equivalence as well. The claim is proved.

So far we have shown that the pushout of a map
from $J$ along any {\bf M}-functor
is in ${\rm cof}(I)\cap {\rm W}$. Since a
transfinite composition of weak equivalences 
of {\bf M} is a weak equivalence, we readily obtain 
that $cell(J)\subset {\rm cof}(I)\cap {\rm W}$. Thus,
condition $nc1$ is checked.

Now, putting all the three steps together we obtain 
the desired model structure on {\bf M}\text{-}{\bf Cat}.

\emph{Step 4}. Suppose that {\bf M} is right proper. Using the 
explicit construction of pullbacks in {\bf M}\text{-}{\bf Cat}, 
the description of the fibrations between fibrant objects 
and \cite[Lemma 9.4]{Bo}, we conclude that the model
structure on {\bf M}\text{-}{\bf Cat} is right proper.
\end{proof}
\begin{lem}
Let $\mathcal{A}$ be a cofibrant simplicial category. Then 
for each $a\in Ob(\mathcal{A})$ the simplicial monoid 
$a^{\ast}\mathcal{A}=\mathcal{A}(a,a)$ is cofibrant.
\end{lem}
\begin{proof}
Let $S=Ob(\mathcal{A})$. $\mathcal{A}$ is cofibrant in 
{\bf S}\text{-}{\bf Cat} if and only if it is cofibrant as 
an object of ${\bf S}\text{-}{\bf Cat}(S)$. The cofibrant 
objects of ${\bf S}\text{-}{\bf Cat}(S)$ are characterized 
in \cite[7.6]{DK}: they are the retracts of free simplicial 
categories. Therefore it suffices to prove that if 
$\mathcal{A}$ is a free simplicial category then 
$a^{\ast}\mathcal{A}$ is a free simplicial category
for all $a\in S$.  There is a full and faithful functor
$\varphi:{\bf S}\text{-}{\bf Cat}\rightarrow {\bf Cat}^{\Delta^{op}}$ 
given by $Ob(\varphi(\mathcal{A})_{n})=Ob(\mathcal{A})$ for all $n\geq 0$
and $\varphi(\mathcal{A})_{n}(a,a')=\mathcal{A}(a,a')_{n}$.
Recall \cite[4.5]{DK} that $\mathcal{A}$ is a
free simplicial category if and only if ($i$) for all $n\geq 0$ the category
$\varphi(\mathcal{A})_{n}$ is a free category on a graph
$G_{n}$, and ($ii$) for all epimorphisms $\alpha:[m]\rightarrow [n]$
of $\Delta$, $\alpha^{\ast}:\varphi(\mathcal{A})_{n}\rightarrow
\varphi(\mathcal{A})_{m}$ maps $G_{n}$ into $G_{m}$.

Let $a\in S$. The category $\varphi(a^{\ast}\mathcal{A})_{n}$ 
is a full subcategory of $\varphi(\mathcal{A})_{n}$ 
with object set $\{a\}$, hence it is free as well. A set
$G^{a^{\ast}\mathcal{A}}_{n}$ of generators 
can be described as follows. An element of 
$G^{a^{\ast}\mathcal{A}}_{n}$ is a path from
$a$ to $a$ in $\varphi(\mathcal{A})_{n}$ such 
that every arrow in the path belongs to $G_{n}$ 
and there is at most one arrow in the path
with source and target $a$. The fact that 
$\varphi(a^{\ast}\mathcal{A})_{n}$ is indeed 
freely generated by $G^{a^{\ast}\mathcal{A}}_{n}$
follows from Lemma 2.5 and its proof. 
Since every epimorphism $\alpha:[m]\rightarrow [n]$ 
of $\Delta$ has a section,  $\alpha^{\ast}$ maps 
$G^{a^{\ast}\mathcal{A}}_{n}$ 
into $G^{a^{\ast}\mathcal{A}}_{m}$.
\end{proof}
\begin{lem}
A full subcategory of a free category is free.
\end{lem}
\begin{proof}
Let $F(G)$ be a free category generated
by a graph $G=(G_{1}\rightrightarrows G_{0})$.
An arrow $f$ of $F(G)$ is a generator if and only
if $f$ is \emph{indecomposable} ($f$ is not a unit 
and $f=vu$ implies $v$ or $u$ is a unit).
Let $C$ be a full subcategory
of $F(G)$ with $Ob(C)=C_{0}\subset G_{0}$.
If $x,y\in C_{0}$, let us say that a path
$(x_{1},f_{1},...,f_{n-1},x_{n})\colon x\to y$
in the graph $G$ is $C_{0}$-\emph{free} if 
$target(f_{i})\notin C_{0}$ for $1\leq i< n$. 
Let $G_{1}'$ be the set of $C_{0}$-free paths.
It is easy to see that every 
arrow of $C$ can be uniquely written as a 
finite composition of $C_{0}$-free paths, 
so that $C$ is freely generated by
the graph $(G_{1}'\rightrightarrows C_{0})$.
\end{proof}
\begin{remark}
The class of cofibrations of the model
category constructed in Theorem 2.3 
can be given an explicit description 
\cite[Section 4.2]{St1}.
\end{remark}
\begin{remark}
We noticed during the proof of 
Theorem 2.3 that our result is almost 
independent on Theorem 2.2, only a 
relaxed version of \cite[Proposition 2.3]{Be}
being needed. In particular, taking ${\bf M}={\bf S}$
in Theorem 2.3 results in a weaker version 
of Bergner's result. However, using the 
fact that {\bf S} has a monoidal fibrant 
replacement functor that preserves fibrations, 
the full Theorem 2.2 can be recovered.
\end{remark}
\begin{remark}
One can change the assumptions of 
Theorem 2.3 and the recognition principle 
used in its proof to obtain a similar outcome. 
For example, let {\bf M} be a cofibrantly 
generated monoidal {\bf S}-model category 
having cofibrant unit and which satisfies the 
monoid axiom. Suppose furthermore that 

$(a)$ a transfinite composition of weak equivalences 
of {\bf M} is a weak equivalence,

$(b)$ {\bf M} satisfies the technical condition of 
\cite[Theorem 2.1]{Ho1}, and 

$(c)$ in the Quillen pair 
$F:{\bf S}\rightleftarrows {\bf M}:G$ 
guaranteed by the definition, the functor 
$G$ preserves weak equivalences. 

Then \cite[Theorem 11.3.1]{Hi} can be 
used to show that {\bf M}\text{-}{\bf Cat} 
admits a cofibrantly generated model category
structure in which the weak equivalences are 
the DK-equivalences and the fibrations are the 
DK-fibrations. The proof proceeds in th same way
as the proof of Theorem 2.3, Step 3 remains 
unchanged but Step 2 requires suitable modifications.
Condition $(b)$ can be relaxed, it was stated in this 
form in order to include examples such as compactly 
generated spaces \cite{Ho1}.
\end{remark}

\section{Pushouts along full and faithful functors}
A result of R. Fritsch and D.M. Latch \cite[Proposition 5.2]{FL}
says that the pushout of a full and faithful functor 
is full and faithful. The purpose of this section 
is to extend this result to categories enriched
over a monoidal category.

Let $(\mathcal{V},\otimes,I)$ be a cocomplete closed category. 
We denote by $\mathcal{V}$\text{-}{\bf Cat} the category of small 
$\mathcal{V}$-categories and by $\mathcal{V}$\text{-}{\bf Graph}
that of small $\mathcal{V}$-graphs. A $\mathcal{V}$-functor, or a 
map of $\mathcal{V}$-graphs, that is locally an isomorphism 
(Section 2.2) is said to be \emph{full and faithful}. If $S$ is a set, we denote by 
$\mathcal{V}$\text{-}{\bf Cat}$(S)$ (resp. $\mathcal{V}$\text{-}{\bf Graph}$(S)$)
the category of small $\mathcal{V}$-categories (resp. $\mathcal{V}$-graphs)
with fixed set of objects $S$. The category $\mathcal{V}$\text{-}{\bf Graph}$(S)$
is a monoidal category with monoidal product $\square_{S}$ 
and unit which we denote by $\mathcal{I}_{S}$. 
\begin{proposition}
Let $\mathcal{A}$, $\mathcal{B}$ and $\mathcal{C}$ be three small
$\mathcal{V}$-categories and let $i:\mathcal{A}\hookrightarrow
\mathcal{B}$ be a full and faithful inclusion. Then in the pushout
diagram of $\mathcal{V}$-categories
\[
\xymatrix{
\mathcal{A} \ar[r]^{i} \ar[d]_{f} & \mathcal{B} \ar[d]^{g}\\
\mathcal{C} \ar[r]^{i'} & \mathcal{D}\\
}
  \]
the map $i':\mathcal{C}\rightarrow \mathcal{D}$ is a full and
faithful inclusion.
\end{proposition}
\begin{proof}
We shall construct $\mathcal{D}$ explicitly, as was done in the
proof of \cite[Proposition 5.2]{FL}. On objects we put
$Ob(\mathcal{D})=Ob(\mathcal{C})\sqcup
(Ob(\mathcal{B})-Ob(\mathcal{A}))$ and
$\mathcal{D}(p,q)=\mathcal{C}(p,q)$ if $p,q\in Ob(\mathcal{C})$. For
$p\in Ob(\mathcal{C})$ and $q\in (Ob(\mathcal{B})-Ob(\mathcal{A}))$
we define $$\mathcal{D}(p,q)=\int^{x\in
Ob(\mathcal{A})}\mathcal{B}(x,q)\otimes \mathcal{C}(p,f(x))$$ For
$p\in (Ob(\mathcal{B})-Ob(\mathcal{A}))$ and $q\in Ob(\mathcal{C})$
we define
$$\mathcal{D}(p,q)=\int^{x\in
Ob(\mathcal{A})}\mathcal{C}(f(x),q)\otimes \mathcal{B}(p,x)$$ For
$p,q\in (Ob(\mathcal{B})-Ob(\mathcal{A}))$ we define
$\mathcal{D}(p,q)$ to be the pushout
\[
\xymatrix{ \int^{x\in Ob(\mathcal{A})}\mathcal{B}(x,q)\otimes
\mathcal{B}(p,x) \ar[r] \ar[d] & \int^{x\in
Ob(\mathcal{A})}\int^{y\in Ob(\mathcal{A})}\mathcal{B}(x,q)\otimes
\mathcal{C}(f(y),f(x))\otimes \mathcal{B}(p,y) \ar[d]\\
\mathcal{B}(p,q) \ar[r] & \mathcal{D}(p,q)\\
}
  \] 
We shall describe a way to see that, with the above definition, 
$\mathcal{D}$ is indeed a $\mathcal{V}$-category.

Let $(\mathcal{B}-\mathcal{A})^{+}$ be the preorder with objects all
finite subsets $S\subset Ob(\mathcal{B})-Ob(\mathcal{A})$, ordered
by inclusion. For $S\in (\mathcal{B}-\mathcal{A})^{+}$, let
$\mathcal{A}_{S}$ be the full sub-$\mathcal{V}$-category of
$\mathcal{B}$ with objects $Ob(\mathcal{A})\cup S$. Then
$\mathcal{B}=\underset{(\mathcal{B}-\mathcal{A})^{+}}\lim
\mathcal{A}_{S}$. On the other hand, a filtered colimit of full and
faithful inclusions of $\mathcal{V}$-categories is a full and
faithful inclusion. This is because the forgetful functor from $\mathcal{V}$\text{-}{\bf Cat}
to $\mathcal{V}$\text{-}{\bf Graph} preserves filtered colimits \cite[Corollary 3.4]{KL}
and a filtered colimit of full and faithful inclusions of
$\mathcal{V}$-graphs is a full and faithful inclusion. Therefore one
can assume from the beginning that
$Ob(\mathcal{B})=Ob(\mathcal{A})\cup \{q\}$, where $q\not \in
Ob(\mathcal{A})$.

\emph{Case 1: $f$ is full and faithful.} In this case the pushout
giving $\mathcal{D}(q,q)$ is simply $\mathcal{B}(q,q)$, all the
other formulas remain unchanged. Then to show that $\mathcal{D}$ is
a $\mathcal{V}$-category is straightforward.

\emph{Case 2: $f$ is the identity on objects.} The map $i$ induces
an adjoint pair
$$i_{!}:\mathcal{V}\text{-}{\bf Cat}(Ob(\mathcal{A}))\rightleftarrows
\mathcal{V}\text{-}{\bf Cat}(Ob(\mathcal{B})):i^{\ast}$$ One has
$$i_{!}\mathcal{A}(a,a') =
  \begin{cases}
\mathcal{A}(a,a'), & \text{if }  a,a'\in Ob(\mathcal{A}),\\
\emptyset, & \text{otherwise},\\
I, & \text{if } a=a'=q,\\
\end{cases} $$
and $i$ factors as $\mathcal{A} \rightarrow i_{!}\mathcal{A}
\rightarrow \mathcal{B}$, where $i_{!}\mathcal{A} \rightarrow
\mathcal{B}$ is the obvious map in
$\mathcal{V}\text{-}{\bf Cat}(Ob(\mathcal{B}))$. Then the original
pushout can be computed using the pushout diagram
\[
\xymatrix{
i_{!}\mathcal{A} \ar[r] \ar[d]_{i_{!}f} & \mathcal{B} \ar[d]\\
i_{!}\mathcal{C} \ar[r] & \mathcal{D}\\
}
  \]
in $\mathcal{V}\text{-}{\bf Cat}(Ob(\mathcal{B}))$. Next, we 
claim that $\mathcal{D}$ can be calculated as the pushout, in the
category $_\mathcal{B}Mod_{\mathcal{B}}$ of
$(\mathcal{B},\mathcal{B})$-bimodules in 
$$(\mathcal{V}\text{-}{\bf Cat}(Ob(\mathcal{B})), 
\square_{Ob(\mathcal{B})}, \mathcal{I}_{Ob(\mathcal{B})})$$
of the diagram
\[
\xymatrix{ \mathcal{B}\square_{i_{!}\mathcal{A}}\mathcal{B}
\ar[rrr]^{
\mathcal{B}\square_{i_{!}\mathcal{A}}i_{!}f\square_{i_{!}\mathcal{A}}\mathcal{B}}
\ar[d] & & &
\mathcal{B}\square_{i_{!}\mathcal{A}}i_{!}\mathcal{C}\square_{i_{!}\mathcal{A}}\mathcal{B} \ar[d]^{m}\\
\mathcal{B} \ar[rrr] & & & \mathcal{D}\\
}
  \]
For this we have to show that $\mathcal{D}$ is a monoid in
$_\mathcal{B}Mod_{\mathcal{B}}$. We first show that
$\mathcal{B}\square_{i_{!}\mathcal{A}}i_{!}\mathcal{C}\square_{i_{!}\mathcal{A}}\mathcal{B}$
is a monoid in $_\mathcal{B}Mod_{\mathcal{B}}$. There is a canonical
isomorphism
$$i_{!}\mathcal{C}\square_{i_{!}\mathcal{A}}i_{!}\mathcal{C}\cong
i_{!}\mathcal{C}\square_{i_{!}\mathcal{A}}\mathcal{B}
\square_{i_{!}\mathcal{A}}i_{!}\mathcal{C}$$ of
$(i_{!}\mathcal{A},i_{!}\mathcal{A})$-bimodules which is best seen
pointwise, using coends. This provides a multiplication for
$\mathcal{B}\square_{i_{!}\mathcal{A}}i_{!}\mathcal{C}\square_{i_{!}\mathcal{A}}\mathcal{B}$
which is again best seen to be associative by working pointwise,
using coends. To define a multiplication for $\mathcal{D}$ consider
the cube diagrams
\[
   \xymatrix{
\mathcal{B}\cdot
i_{!}\mathcal{A}\cdot\mathcal{B}\cdot_{\mathcal{B}}\mathcal{B}\cdot
i_{!}\mathcal{A}\cdot\mathcal{B} \ar[rr] \ar[dr] \ar[dd] & &
\mathcal{B}\cdot_{\mathcal{B}}\mathcal{B}\cdot i_{!}\mathcal{A}\cdot\mathcal{B} \ar[drr] \ar[dd] \\
& \mathcal{B}\cdot
i_{!}\mathcal{C}\cdot\mathcal{B}\cdot_{\mathcal{B}}\mathcal{B}\cdot
i_{!}\mathcal{A}\cdot\mathcal{B} \ar[rrr] \ar[dd] & & &
\mathcal{D}\cdot_{\mathcal{B}}\mathcal{B}\cdot
i_{!}\mathcal{A}\cdot\mathcal{B} \ar[dd]\\
\mathcal{B}\cdot
i_{!}\mathcal{A}\cdot\mathcal{B}\cdot_{\mathcal{B}}\mathcal{B}\cdot
i_{!}\mathcal{C}\cdot\mathcal{B}
\ar[rr] \ar[dr] & & \mathcal{B}\cdot_{\mathcal{B}}\mathcal{B}\cdot i_{!}\mathcal{C}\cdot\mathcal{B} \ar[drr]\\
& \mathcal{B}\cdot
i_{!}\mathcal{C}\cdot\mathcal{B}\cdot_{\mathcal{B}}\mathcal{B}\cdot
i_{!}\mathcal{C}\cdot\mathcal{B}
 \ar[rrr] & & & \mathcal{D}\cdot_{\mathcal{B}}\mathcal{B}\cdot i_{!}\mathcal{C}\cdot\mathcal{B}\\
 }
   \]
and
\[
   \xymatrix{
\mathcal{B}\cdot_{\mathcal{B}}\mathcal{B}\cdot
i_{!}\mathcal{A}\cdot\mathcal{B} \ar[rr] \ar[dr] \ar[dd] & &
 \mathcal{B}\cdot_{\mathcal{B}}\mathcal{B} \ar[drr] \ar[dd] \\
& \mathcal{D}\cdot_{\mathcal{B}}\mathcal{B}\cdot
i_{!}\mathcal{A}\cdot\mathcal{B}
 \ar[rrr] \ar[dd] & & & \mathcal{D}\cdot_{\mathcal{B}}\mathcal{B} \ar[dd]\\
 \mathcal{B}\cdot_{\mathcal{B}}\mathcal{B}\cdot i_{!}\mathcal{C}\cdot\mathcal{B}
  \ar[rr] \ar[dr] & & \mathcal{B}\cdot_{\mathcal{B}}\mathcal{D} \ar[drr]\\
& \mathcal{D}\cdot_{\mathcal{B}}\mathcal{B}\cdot
i_{!}\mathcal{C}\cdot\mathcal{B}
 \ar[rrr] & & & \mathcal{D}\cdot_{\mathcal{B}}\mathcal{D}\\
 }
   \]
For space considerations we have suppressed tensors (always over
$i_{!}\mathcal{A}$, unless explicitly indicated) from notation. The
right face of the first cube is the same as the left face of the
latter cube. Let $PO_{1}$ (resp. $PO_{2}$) be the pushout of the
left (resp. right) face of the first cube diagram. Let $PO_{3}$ be
the pushout of the right face of the second cube diagram. We have
pushout digrams
\[
\xymatrix{ PO_{1} \ar[r] \ar[d] & PO_{2} \ar[r] \ar[d] & PO_{3}
\ar[d]\\
\mathcal{B}\cdot i_{!}\mathcal{C}\cdot
\mathcal{B}\cdot_{\mathcal{B}}\mathcal{B}\cdot
i_{!}\mathcal{C}\cdot\mathcal{B} \ar[r] &
\mathcal{D}\cdot_{\mathcal{B}}\mathcal{B}\cdot
i_{!}\mathcal{C}\cdot\mathcal{B} \ar[r] &
\mathcal{D}\cdot_{\mathcal{B}}\mathcal{D}
 }
  \]
Using these pushouts and the fact that
$\mathcal{B}\square_{i_{!}\mathcal{A}}i_{!}
\mathcal{C}\square_{i_{!}\mathcal{A}}\mathcal{B}$
is a monoid one can define in a canonical way a map
$\mu:\mathcal{D}\cdot_{\mathcal{B}}\mathcal{D}\rightarrow
\mathcal{D}$. We omit the long verification that $\mu$ gives
$\mathcal{D}$ the structure of a monoid. The map $\mu$ was
constructed in such a way that $m$ becomes a morphism of monoids.
The fact that $\mathcal{D}$ has the universal property of the
pushout in the category $\mathcal{V}\text{-}{\bf Cat}(Ob(\mathcal{B}))$
follows from its definition.

\emph{Case 3: $f$ is arbitrary.} Let $u=Ob(f)$. We factorize $f$ as
$\mathcal{A}\overset{f^{u}} \rightarrow u^{\ast}\mathcal{C}
\rightarrow \mathcal{C}$, where $Ob(u^{\ast}\mathcal{C})=Ob(\mathcal{A})$,
$u^{\ast}\mathcal{C}(a,a')=\mathcal{C}(fa,fa')$ and $f^{u}$ 
is the obvious map, and take consecutive pushouts:
\[
\xymatrix{
\mathcal{A} \ar[r]^{i} \ar[d]_{f^{u}} & \mathcal{B} \ar[d]\\
u^{*}\mathcal{C} \ar[d] \ar[r] & \mathcal{A'} \ar[d]\\
\mathcal{C} \ar[r] & \mathcal{D}\\
}
  \]
Now apply Case 2  to $f^{u}$ and Case 1
to $u^{\ast}\mathcal{C}\rightarrow \mathcal{C}$.
\end{proof}

\section{Application: left Bousfield localizations 
of categories of monoids}
This section was motivated by the paragraph `As we 
mentioned above,...in general.' on page 111 of \cite{Ho2}.

\subsection{The problem}
Let {\bf M} be a (suitable) monoidal model category, 
$L${\bf M} a left Bousfield localization of {\bf M}
which is itself a monoidal model category and 
$Mon({\bf M})$ the category of monoids in {\bf M}. 
The problem is to induce on $Mon({\bf M})$ a model 
category structure somehow related to $L${\bf M}.
As pointed out in \cite{Ho2}, such a model structure 
exists if, for example, $(a)$ $L${\bf M} 
satisfies the monoid axiom or $(b)$ $Mon({\bf M})$ 
has a suitable left proper model category structure. 
In order for $(a)$ to be fulfilled one needs to know 
the (generating) trivial cofibrations of $L${\bf M}.
However, it often happens that one does not have 
an explicit description of them. For $(b)$, the category of 
monoids in a monoidal model category is rarely known to 
be left proper (it is left proper when the underlying 
model category has all objects cofibrant, for instance,
which seems to us too restrictive to work with).

\subsection{Our solution}
We shall propose below a solution to the above 
problem. We shall reduce the verification of the monoid 
axiom for $L${\bf M} to a smaller---and hopefully more 
tractable in practice, set of maps and we shall avoid left 
properness by using Theorem 0.1 via Proposition 1.3.
The model category theoretical framework 
will be the `combinatorial' counterpart of the one 
of \cite[Section 8]{Ho2}. 

It will be clear that the method could potentially 
be applied to other structures than monoids.

\subsubsection{Recollections on enriched 
left Bousfield localization}
We recall some facts from \cite{Ba}.
Let $\mathcal{V}$ be a monoidal model 
category and {\bf M} a model $\mathcal{V}$-category 
with tensor, hom and cotensor denoted by
$$-\ast-:\mathcal{V}\times {\bf M}\rightarrow {\bf M}$$
$$Map(-,-):{\bf M}^{op}\times {\bf M}\rightarrow \mathcal{V}$$
$$(-)^{(-)}:\mathcal{V}^{op}\times {\bf M}\rightarrow {\bf M}$$
Let $S$ be a set of maps of {\bf M} between cofibrant objects.
\begin{definition}
A fibrant object $W$ of {\bf M} is $S$-\emph{local} if for every
$f\in S$ the map $Map(f,W)$ is a weak equivalence of $\mathcal{V}$.
A map $f$ of {\bf M} is an $S$-\emph{local equivalence}
if for every $S$-local object $W$ and for some (hence any)
cofibrant approximation $\tilde{f}$ to $f$, the map 
$Map(\tilde{f},W)$ is a weak equivalence of $\mathcal{V}$.
\end{definition}
In the previous definition, if the map $Map(\tilde{f},W)$ is 
a weak equivalence of $\mathcal{V}$, then for any other 
cofibrant approximation $\tilde{g}$ to $f$, the map 
$Map(\tilde{g},W)$ is a weak equivalence of 
$\mathcal{V}$ \cite[Proposition 14.6.6(1)]{Hi}.
\begin{theorem}
\cite{Ba} Let $\mathcal{V}$ be a combinatorial monoidal 
model category, {\bf M} a left proper, combinatorial model 
$\mathcal{V}$-category and $S$ a set of maps of {\bf M} 
between cofibrant objects. Suppose that $\mathcal{V}$ 
has a set of generating cofibrations with cofibrant domains.

Then the category {\bf M} admits a left proper, combinatorial 
model category structure, denoted by $L_{S}{\bf M}$, with 
the class of $S$-local equivalences as weak equivalences and 
the same cofibrations as the given ones. The fibrant objects of 
$L_{S}{\bf M}$ are the $S$-local objects. 
$L_{S}{\bf M}$ is a model $\mathcal{V}$-category. 

Suppose that, moreover, {\bf M} is a monoidal model 
$\mathcal{V}$-category which has a set of generating 
cofibrations with cofibrant domains. Let us denote by 
$\otimes$ the monoidal product on {\bf M}. 
If $X\otimes f$ is an $S$-local equivalence
for every $f\in S$ and every $X$ belonging to the domains 
and codomains of the generating cofibrations of {\bf M}, 
then $L_{S}{\bf M}$ is a monoidal model 
$\mathcal{V}$-category.
\end{theorem}

\subsubsection{The $S$-extended monoid axiom}
Let $\mathcal{V}$ be a monoidal model 
category and {\bf M} a monoidal model 
$\mathcal{V}$-category with monoidal product $\otimes$ 
and tensor, hom and cotensor denoted as in 4.2.1.
If $i:K\rightarrow L$ is a map of $\mathcal{V}$ and 
$f:A\rightarrow B$ a map of {\bf M}, we denote by $i\ast' f$
the canonical map $$L\ast A\underset{K\ast A}\cup K\ast 
B\rightarrow L\ast B$$ Let $S$ be a set of maps of {\bf M}
between cofibrant objects. For every $f\in S$, let 
$f=v_{f}u_{f}$ be a factorization of $f$ as a cofibration 
$u_{f}$ followed by a weak equivalence $v_{f}$;
a concrete one is the mapping cylinder factorization.
\begin{definition}
We say that {\bf M} satisfies the $S$-\emph{extended monoid 
axiom} if, in the notation of \cite[Section 3]{SS1}, every 
map in $$(\{{\rm trivial \ cofibrations \ of\ {\bf M}}\}
\cup (\{{\rm cofibrations\ of\ \mathcal{V}}\}\ast' u_{f})_{f\in S})
\otimes {\bf M}\text{-}{\rm cof_{reg}}$$ is an $S$-local equivalence.
\end{definition}
As usual \cite[Lemma 3.5(2)]{SS1}, if $\mathcal{V}$ and 
{\bf M} are cofibrantly generated and every map in 
$$(\{{\rm generating \ trivial \ cofibrations\ of\ {\bf M}}\}\cup 
(\{{\rm generating \ cofibrations \ of\ \mathcal{V}}\}\ast' u_{f})_{f\in S})
\otimes {\bf M}\text{-}{\rm cof_{reg}}$$ is an $S$-local 
equivalence, then the $S$-extended monoid axiom holds.

Let $Mon({\bf M})$ be the category
of monoids in {\bf M} and let $$T:{\bf M}
\rightleftarrows Mon({\bf M}):U$$ be 
the free-forgetful adjunction.
\begin{definition}
A monoid $M$ in {\bf M} is $TS$-\emph{local} if $U(M)$
is $S$-local. A map $f$ of monoids in {\bf M} is a 
$TS$-\emph{local equivalence} if $U(f)$ is an 
$S$-local equivalence.
\end{definition}
\begin{theorem}
Let $\mathcal{V}$ be a combinatorial monoidal model 
category having a set of generating cofibrations with cofibrant 
domains. Let {\bf M} be a left proper, combinatorial monoidal 
model $\mathcal{V}$-category which has a set of generating 
cofibrations with cofibrant domains. Let us denote by 
$\otimes$ the monoidal product on {\bf M}.
Let $S$ be a set of maps of {\bf M} between cofibrant objects. 
 Suppose that $X\otimes f$ is an $S$-local equivalence for every 
$f\in S$ and every $X$ belonging to the domains and codomains 
of the generating cofibrations of {\bf M} and that {\bf M}
satisfies the $S$-extended monoid axiom.

Then the category $Mon({\bf M})$ admits a combinatorial 
model category structure with $TS$-local equivalences 
as weak equivalences and with 
$T(\{{\rm cofibrations\ of \ {\bf M}}\})$ as cofibrations. 
The fibrant objects are the $TS$-local objects.
\end{theorem}
\begin{proof}
We shall apply Theorem 0.1 via Proposition 1.3. 
We take $\mathcal{E}$ to be $Mon({\bf M})$, {\rm W} 
to be the class of $TS$-local equivalences, $I$ to be the 
set $T(\{{\rm generating \ cofibrations \ of\ {\bf M}}\})$ 
and $J$ to be $$T(\{{\rm generating \ trivial \ cofibrations \ of \ {\bf M}}\} 
\cup \{{\rm generating \ cofibrations\ of\ \mathcal{V}}\ast' u_{f}\}_{f\in S})$$
Notice that a map $g$ of monoids in {\bf M} belongs to 
${\rm inj}(T(\{{\rm generating \ cofibrations\ of \ {\bf M}}\}))$
if and only if $U(g)$ belongs to 
${\rm inj}(\{{\rm generating \ cofibrations\ of \ {\bf M}}\})$
if and only if $U(g)$ is a trivial fibration of {\bf M}. Therefore
condition $c1$ from Proposition 1.3 holds.

We claim that a monoid $M$ in {\bf M} is naively 
fibrant if and only if $M$ is $TS$-local. We may
assume without loss of generality that $U(M)$ is 
fibrant. We observe that if $i$ is any map of 
$\mathcal{V}$ and $f\in S$, then $M$ has the right
lifting property with respect to $T(i\ast' u_{f})$
if and only if $Map(u_{f},U(M))$ has the 
right lifting property with respect to $i$.
Since $Map(v_{f},U(M))$ is a weak equivalence 
of $\mathcal{V}$ and $Map(u_{f},U(M))$
is a fibration of $\mathcal{V}$, the claim 
follows from this observation.

Let now $g$ be a map of monoids in {\bf M}
between $TS$-local monoids such that $g$ is 
both a $TS$-local equivalence and a naive 
fibration. Then $U(g)$ is an $S$-local
equivalence between $S$-local objects,
so $U(g)$ is a weak equivalence. $U(g)$
is also a fibration, therefore condition 
$nc2$ from Proposition 1.3 holds.

Condition $nc1$ from Proposition 1.3
is guaranteed by the $S$-extended monoid 
axiom and \cite[Proof of Lemma 6.2]{SS1}.
\end{proof}

{\bf Acknowledgements.} We are deeply indebted to 
Andr\'{e} Joyal for many useful discussions and suggestions.

\end{document}